\documentclass[12pt]{amsart}

\usepackage{amssymb,latexsym,amsmath,amsthm,amscd}

\usepackage[all]{xy}
\xyoption{arc}

\usepackage{graphicx}
\usepackage{cmap}
\usepackage{ifthen}

\usepackage{enumerate}

\usepackage{float}

\usepackage{url}
\usepackage{hyperref}

\newtheorem{thm}{Theorem}[section]
\newtheorem*{thm*}{Theorem}

\newtheorem{conj*}{Conjecture}

\newtheorem*{lemma*}{Lemma}

\newtheorem*{prop*}{Proposition}

\newtheorem{claim}[thm]{Claim}

\newtheorem*{fact}{Fact}

\theoremstyle{definition}
\newtheorem{df}[thm]{Definition}

\theoremstyle{remark}
\newtheorem*{rmk}{Remark}

\newtheorem*{claim*}{Claim}

\newcommand{\bbQ}{\mathbb{Q}}

\newcommand{\bbF}{\mathbb{F}}
\newcommand{\bbC}{\mathbb{C}}
\newcommand{\bbZ}{\mathbb{Z}}

\newcommand{\bbR}{\mathbb{R}}

\newcommand{\QQ}{\bbQ}

\newcommand{\RR}{\bbR}
\newcommand{\CC}{\bbC}
\newcommand{\FF}{\bbF}

\newcommand{\ZZ}{\bbZ}

\newcommand{\cH}{\mathcal{H}}

\DeclareMathOperator{\Aut}{Aut}

\DeclareMathOperator{\tr}{tr}

\DeclareMathOperator{\ord}{ord}

\DeclareMathOperator{\SL}{SL}
\DeclareMathOperator{\PSL}{PSL}

\DeclareMathOperator*{\mysum}{\sum}

\newcommand{\emphh}[2][ ]{%
\ifthenelse{\equal{#1}{ }}{\index{default}{#2} {\emph{#2}}}{\index{default}{#1@#2} {\emph{#2}}}%
}

\def\sumprime{\mathop{\sum{\raise3pt\hbox{${}'$}}}} 

\newcommand{\tto}[1]{%
\ifthenelse{\equal{#1}{}}{\to}{\stackrel{#1}{\to}}}

\newcommand{\fixme}[1]{}

\begin{document}

\title[On modular forms and representations of $\PSL_2(\FF_q)$]{On modular forms of weight 2 and \\ representations of $\PSL_2(\FF_q)$}

\author{Luiz Kazuo Takei}

\address{Department of Mathematics and Statistics, McGill University, Montr\'eal, QC, Canada}

\email{takei@math.mcgill.ca}


\maketitle

\vspace{\baselineskip}

\begin{abstract}
	This is essentially a translated (and explained) version of \cite{hecke1930}, where Hecke shows, for a prime $q$, a relation between the class number $h(-q)$ of $\QQ(\sqrt{-q})$ and the representation of $\PSL_2(\FF_q)$ on the space of holomorphic differentials of $X(q)$.
\end{abstract}

\section{Introduction}

In 1930 Hecke (\cite{hecke1930}) proved an interesting result concerning the representation of $\PSL_2(\FF_q)$ on the space of holomorphic differentials of $X(q)$. Analyzing the character table of this group we see two `special' irreducible representations (here denoted $\pi_+$ and $\pi_-$). If we call $m_+$ and $m_-$ the multiplicity of $\pi_+$ and $\pi_-$ in that representation, then Hecke showed that $m_+ - m_- = h(-q)$. In this paper we present a translation and explanation of the ideas contained in \cite{hecke1930}.

We attempt to use the same notation as Hecle whenever it is defined, otherwise we use standard notation. A notable exception is the choice of notation for the irreducible representations of $\PSL_2(\FF_q)$.

We assume familiarity with representation theory, Riemann surfaces and some knowledge of modular forms at an introductory level (for instance, the first chapters of \cite{dia&shu2005}). When needed, we refer to results in other topics as well.

\section{Fixing notations}

In these notes, $q$ will denote a prime number (which we assume to satisfy $q \equiv 3 \pmod{4}$ and $q > 3$).
\[
	\Gamma(q) = \left\{ \begin{pmatrix} a & b \\ c & d \end{pmatrix} \in \SL_2(\ZZ) \ \bigg| \ \begin{pmatrix} a & b \\ c & d \end{pmatrix} \equiv \begin{pmatrix} 1 & 0 \\ 0 & 1 \end{pmatrix} \pmod q \right\}
\]
\[	
	\Gamma_1(q) = \left\{ \begin{pmatrix} a & b \\ c & d \end{pmatrix} \in \SL_2(\ZZ) \ \bigg| \ \begin{pmatrix} a & b \\ c & d \end{pmatrix} \equiv \begin{pmatrix} 1 & * \\ 0 & 1 \end{pmatrix} \pmod q \right\}
\]
\[
	M_2(\Gamma(q)) := \textrm{space of modular forms of weight $2$ with respect to } \Gamma(q)
\]
(as defined in \cite{dia&shu2005}; in particular, our modular forms are holomorphic on $\cH$ and at the cusps)

\vspace{.5cm}

We have a well-known (right) action of $\SL_2(\ZZ)$ on $M_2(\Gamma(q))$: if $\gamma = \begin{pmatrix} a & b \\ c & d \end{pmatrix} \in \SL_2(\ZZ)$ and $\varphi \in M_2(\Gamma(q))$, then
\[
	(\varphi[\gamma]_2) (z) := (cz+d)^{-2} \varphi(\gamma z)
\]
where $\gamma z := \frac{az+b}{cz+d}$.

Since the elements of $M_2(\Gamma(q))$ are invariant under the action of $\Gamma(q)$ and $-I$, the original action induces an action of $\PSL_2(\FF_q) = \dfrac{\SL_2(\ZZ)}{\{\pm I \} \cdot \Gamma(q)}$.

\vspace{.5cm}

Let $\zeta := \exp\left(\frac{2 \pi i}{q}\right)$ and $P = \begin{pmatrix} 1 & 1 \\ 0 & 1 \end{pmatrix} \in PSL_2(\FF_q)$.

We define the $\CC$-vector space $V := \left\{ f \in M_2(\Gamma(q)) \ | \ f[P]_2 = \zeta f \right\}$ and denote $z := \dim_{\CC} V$.

We also define the following modular curves (which are Riemann surfaces): $X(q) = \Gamma(q) \backslash \cH^*$ and $X_1(q) = \Gamma_1(q) \backslash \cH^*$ (where $\cH^*$ is the union of the upper half-plane and the cusps).

Finally, throughout this paper, if $\pi$ is a representation of a group $G$ on a vector space $V$, $\pi(g)$ denotes the element in $GL(V)$ or its trace depending on the context (sometimes we write $\tr(\pi(g))$ so that it is clear what we mean).

\section{First Remarks}

Note that the action of $\PSL_2(\FF_q)$ on $M_2(\Gamma(q))$ induces a representation (since it is a right action, the representation is given by $\gamma \cdot \varphi = \varphi[\gamma^{-1}]_2$). 

Recall that $M_2(\Gamma(q)) = E_2(\Gamma(q)) \oplus S_2(\Gamma(q))$, where $E_2(\Gamma(q))$ is the Eisenstein space and $S_2(\Gamma(q))$ is the space of cusp forms. This decomposition is also a decomposition of $M_2(\Gamma(q))$ as a representation of $\PSL_2(\FF_q)$. Moreover, $S_2(\Gamma(q))$ is isomorphic to the space of holomorphic differentials of $X(q)$.

$\PSL_2(\FF_q)$ can be naturally identified with a subgroup of $\Aut(X(q))$ and its action on $S_2(\Gamma(q))$ defined in the previous section is (via these identifications) the action of a subgroup of $\Aut(X(q))$ on the space of holomorphic differentials of $X(q)$.

If $\pi$ denotes the representation from last paragraph, then it is known that $\pi(\gamma) + \overline{\pi(\gamma)} = 2 - t$ where $t$ is the number of fixed points of $\gamma \in \PSL_2(\FF_q) \subseteq \Aut(X(q))$ (this is called the Lefschetz Fixed Point Formula; cf. \cite{farkas&kra1992}). So, if $\pi(\gamma) \in \RR$, we can obtain $\pi(\gamma)$ by computing $(2-t)/2$.

Looking at the character table of $\PSL_2(\FF_q)$ we see that all the irreducible representations have real traces except in the case $q \equiv 3 \pmod{4}$ (where there are two irreducible representations which do not have real traces).

Thus, from now on we assume $q \equiv 3 \pmod{4}$ and $q > 3$ (the case $q = 3$ is simple and can be treated individually). Hecke actually studied all the cases in \cite{hecke1930} (for the interested reader: Hecke defines $\varepsilon = (-1)^{(q-1)/2}$ and deals with both cases simultaneously).

\section{Character Table of $PSL_2(\FF_q)$}

In this section we recall the character table of $\PSL_2(\FF_q)$ (following the presentation given in \cite{casselmana}).


Representations of the conjugacy classes are:
\[
	\begin{array}{cccc}
		\pm \begin{pmatrix} 1 & 0 \\ 0 & 1 \end{pmatrix} & & & \\
		\pm \begin{pmatrix} t & 0 \\ 0 & 1/t \end{pmatrix} & \sim & \pm \begin{pmatrix} 1/t & 0 \\ 0 & t \end{pmatrix} & (t \neq \pm 1) \\
		\pm \begin{pmatrix} a & -b \\ b & a \end{pmatrix} & \sim & \pm \begin{pmatrix} a & b \\ -b & a \end{pmatrix} & (a^2 + b^2 = 1) \\
		\pm \begin{pmatrix} 1 & 1 \\ 0 & 1 \end{pmatrix} & & & \\
		\pm \begin{pmatrix} 1 & -1 \\ 0 & 1 \end{pmatrix} & & & 		
	\end{array}
\]

If we let $E = \FF_q(\sqrt{-1})$, then $N^1_{E / \FF_q}$ embeds in $\SL_2(\FF_q)$ via $a + b \sqrt{-1} \mapsto \begin{pmatrix} a & -b \\ b & a \end{pmatrix}$. Since $N^1_{E / \FF_q} / \{ \pm 1 \}$ has a unique element of order $2$ (namely $\sqrt{-1}$), it also has a unique character of order $2$ (call it $\rho_0$).

The irreducible representations of $\PSL_2(\FF_q)$ are the following:
\begin{itemize}
	\item the trivial representation of dimension $1$
	\item the Steinberg representation
	\item representations $\pi_{\chi}$ parametrized by the characters $\chi$ of the group $\FF_q^{\times} / \{ \pm 1 \}$
	\item representations $\pi_{\rho}$ parametrized by the characters $\rho \neq \rho_0$ of the group $N^1_{E / \FF_q} / \{ \pm 1 \}$
	\item representations $\pi_+$ and $\pi_-$, corresponding to the character $\rho_0$
\end{itemize}

Finally, we give the character table:

\begin{table}[H]
	$
	\begin{array}{| c | c | c | c | c | c | c| }
		\hline & id & St & \pi_{\chi} & \pi_{\rho} & \pi_+ & \pi_- \\ \hline
		I = \begin{pmatrix} 1 & 0 \\ 0 & 1 \end{pmatrix} & 1 & q & q+1 & q-1 & (q-1)/2 & (q-1)/2 \\ \hline
		\begin{pmatrix} t & 0 \\ 0 & 1/t \end{pmatrix} & 1 & 1 & \chi(t) + \chi(1/t) & 0 & 0 & 0 \\ \hline
		\begin{pmatrix} a & -b \\ b & a \end{pmatrix} & 1 & -1 & 0 & -\rho(\epsilon) - \rho(1 / \epsilon) & -\rho_0(\epsilon) & -\rho_0(\epsilon) \\ \hline
		P = \begin{pmatrix} 1 & 1 \\ 0 & 1 \end{pmatrix} & 1 & 0 & 1 & -1 & \overline{\mathfrak{G}} & \mathfrak{G} \\ \hline
		P^{-1} = \begin{pmatrix} 1 & -1 \\ 0 & 1 \end{pmatrix} & 1 & 0 & 1 & -1 & \mathfrak{G} & \overline{\mathfrak{G}} \\ \hline
	\end{array}
	$
	
	\caption{Irreducible representations of $\PSL_2(\FF_q)$}
	\label{tbl:irreps}
\end{table}

\vspace{.5cm}

where $\varepsilon = a + b \sqrt{-1} \in E$,
\[
	\mathfrak{G} = \mysum_{\left( \frac{x}{q} \right) = 1} \exp(2 \pi i x / q) = \mysum_{\left( \frac{x}{q} \right) = 1} \zeta^x
\]
and $\overline{\mathfrak{G}}$ is its complex conjugate.

\section{Computing $z$ using representation theory}

In this section we will compute $z$ viewing $M_2(\Gamma(q))$ as a representation of $\PSL_2(\FF_q)$.

\subsection{General remarks on representations of $\PSL_2(\FF_q)$}

For each representation $\pi$ and each $n \in \{ 0, \dotsc, q-1 \}$, let $p_{\pi}(n)$ denote the multiplicity of $\zeta^n$ viewed as an eigenvalue of $\pi(P^{-1})$. (Sometimes we shall simply write $p$ instead of $p_{\pi}$). With this notation, we have the following:
\[
	f_\pi := \tr(\pi(I)) = p(0) + p(1) + \dotsb + p(q-1)
\]
\[
	\tr(\pi(P)) = p(0) 1 + p(1) \zeta + \dotsb + p(q-1) \zeta^{q-1}
\]
(recall that $\pi(\gamma)$ can be diagonalized for any $\gamma \in \PSL_2(\FF_q)$).

So, given $f_\pi$ and $\tr(\pi(P))$, we can determine $p(0), p(1), \dotsb, p(q-1)$.

We can therefore compute $p(n)$ for each of the irreducible representations (cf. table \ref{tbl:irreps}) obtaining the following:

\vspace*{.1cm}

\begin{center}
	\begin{tabular}{r c l}
		$id$ & & $p(0) = 1$ and $p(n) = 0$ for all $n>0$ \vspace*{.05cm} \\
		$St$ & & $p(n) = 1$ for all $n$ \vspace*{.05cm} \\
		$\pi_\chi$ & & $p(0) = 2, p(1) = \dotsb = p(q-1) = 1$ \vspace*{.1cm} \\
		$\pi_\rho$ & & $p(0) = 0, p(1) = \dotsb = p(q-1) = 1$ \\
		$\pi_+$ & & $p(n) = \left\{ \begin{array}{c c l} 1 & \textrm{, } & \left( \frac{n}{q} \right) = 1  \\ 0 & \textrm{, } & \textrm{ otherwise }  \end{array} \right.$ \\
		$\pi_-$ & & $p(n) = \left\{ \begin{array}{c c l} 1 & \textrm{, } & \left( \frac{n}{q} \right) = -1  \\ 0 & \textrm{, } & \textrm{ otherwise }  \end{array} \right.$
	\end{tabular}
\end{center}

As an example of these computations, let us determine $p_{\pi_+}(n)$ for all $n$. We have the equations
\[
	\frac{q-1}{2} = p(0) + p(1) + \dotsb + p(q-1)
\]
\[
	\begin{array}{c c l}
		\mysum\limits_{\left( \frac{x}{q} \right) = 1} \zeta^x & = & p(0)1 + p(1) \zeta + \dotsb + p(q-1) \zeta^{q-1} \\
		 & = & [p(0) - p(q-1)]1 + [p(1) - p(q-1)] \zeta + \dotsb + [p(q-2) + p(q-1)]\zeta^{q-2}
	\end{array}
\]
Hence, for $0 \leq n \leq q-2$
\[
	\begin{array}{ccl}
		p(n) - p(q-1) & = & \left\{ \begin{array}{ccl}
										1 & , & \left( \frac{n}{q} \right) = 1 \\
										0 & , & \textrm{ otherwise}
									\end{array} \right.
	\end{array}
\]

Now, since there are exactly $(q-1)/2$ squares in $\FF_q^{\times}$ the first equation reads
\[
	\frac{q-1}{2} = (q-1) p(q-1) + \frac{q-1}{2}
\]
and so
\[
	p(q-1) = 0
\]

This yields
\[
	\begin{array}{ccl}
		p(n) & = & \left\{ \begin{array}{ccl}
								1 & , & \left( \frac{n}{q} \right) = 1 \\
								0 & , & \textrm{ otherwise}
							\end{array} \right.
	\end{array}
\]

So, if we have a representation $W$ of $\PSL_2(\FF_q)$ which decomposes as
\[
	W = \alpha \cdot St \ \oplus \ \beta_+ \cdot \pi_+ \ \oplus \ \beta_- \cdot \pi_- \ \oplus \ \mysum_{\chi} \gamma_{\chi} \cdot \pi_{\chi} \ \oplus \ \mysum_{\rho} \delta_{\rho} \cdot \pi_{\rho}
\]
then
\[
	\dim_{\CC} \left\{ w \in W \ | \ P^{-1}w = \zeta w \right\} = \alpha + \beta_+ + \mysum \gamma_{\chi} + \mysum \delta_{\rho}
\]

\subsection{Applying the general remarks to our case}
\label{ssec:applying_gen_rmks}

Now we apply the previous subsection to the representation on $M_2(\Gamma(q))$ (recall that if $\gamma \in \PSL_2(\FF_q)$ and $\varphi \in M_2(\Gamma(q))$, then $\gamma \varphi = \varphi[\gamma^{-1}]_2$).

As we saw earlier, $M_2(\Gamma(q)) = E_2(\Gamma(q)) \ \oplus \ S_2(\Gamma(q))$. So, we may study $E_2(\Gamma(q))$ and $S_2(\Gamma(q))$ separately.

It is known that, as a representation,
\[
	E_2(\Gamma(q)) = St \ \oplus \ 2 \mysum_{\chi} \pi_{\chi}
\]

Let
\begin{equation}
\label{eqn:decomp_of_s2}
	S_2(\Gamma(p)) = x \cdot St \ \oplus \ y_+ \cdot \pi_+ \ \oplus \ y_- \cdot \pi_- \ \oplus \ \mysum_{\chi} u_{\chi} \cdot \pi_{\chi} \ \oplus \ \mysum_{\rho} v_{\rho} \cdot \pi_{\rho}
\end{equation}
be the decomposition of the space of cusp forms (notice $id$ does not occur because of the well known fact that $\dim M_2(\SL_2(\ZZ)) = 0$).

Then, since there are $\frac{q + 1}{4} - 1$ irreducible representations of the form $\pi_{\chi}$, the previous subsection gives us
\[
	z = x + 1 + y_+ + \mysum_{\chi} u_{\chi} + 2 \left( \frac{q + 1}{4} - 1 \right) + \mysum_{\rho}v_{\rho}
\]

Our goal for the rest of this section is to simplify this expression. For this, we will define some notations and use them to help us (these are actually introduced in \cite{hecke1928}).

\vspace{.5cm}

\begin{df} Given the decomposition (\ref{eqn:decomp_of_s2}), we define
\[
	\begin{array}{ccc}
		U = \mysum\limits_{\chi} u_{\chi} & , & V = \mysum\limits_{\rho} v_{\rho} \\
		Y = y_- + y_+ & , & S = Y + 2U + 2V
	\end{array}
\]
\end{df}

\vspace{.5cm}

Using this notation, we obtain
\begin{equation}
\label{eqn:formula_z}
	z = \frac{y_+ - y_-}{2} + \frac{q-1}{2} + \frac{Y}{2} + U + V + x
\end{equation}

\vspace{.5cm}

\begin{df} If $H \leq \PSL_2(\FF_q)$, we define
\[
	Z(H) := \dim_{\CC} \{ f \in S_2(\Gamma(q)) \ | \ f[\gamma]_2 = f, \forall \gamma \in H \}
\]
\end{df}

\vspace{.5cm}

Notice that $Z(H)$ is just the multiplicity of the identity representation of $H$ in $S_2(\Gamma(q))$. Thus, by representation theory,
\begin{equation}
\label{eqn:formula_Z(H)}
    Z(H) = \frac{1}{|H|} \mysum_{h \in H} \left( x St(h) + y_+ \pi_+(h) + y_- \pi_-(h) + \mysum_{\chi} u_{\chi} \pi_{\chi}(h) + \mysum_{\rho} v_{\rho} \pi_{\rho}(h) \right)
\end{equation}

Notice $H_1 := \left\{ \begin{pmatrix} a & -b \\ b & a \end{pmatrix} \in \PSL_2(\FF_q) \ \bigg| \ a, b \in \FF_q \right\}$ is a subgroup of order $(q+1)/2$. Also, $H_2 := \left\{ \begin{pmatrix} t & 0 \\ 0 & 1/t \end{pmatrix} \in \PSL_2(\FF_q) \ \bigg| \ t \in \FF_q^\times \right\}$ is a subgroup of order $(q-1)/2$ (because $\FF_q^\times$ is cyclic of order $q-1$ and $I = -I$ in $\PSL_2(\FF_q)$). This motivates our next definition:

\begin{df} $Z(\frac{q+1}{2}) := Z(H_1)$ \ and \ $Z(\frac{q-1}{2}) := Z(H_2)$.
\end{df}

Using equation (\ref{eqn:formula_Z(H)}), we obtain the following:

\begin{equation}
\label{eqn:specific_formula_Z(H)}
    \begin{array}{c}
        Z(\frac{q-1}{2}) = 3x + Y + 2U + 2V \\ \\
        Z(\frac{q+1}{2}) = x + Y + 2U + 2V
    \end{array}
\end{equation}

Hence,
\begin{equation}
\label{eqn:specific_formula_x_S}
    \begin{array}{c}
        x = \frac{1}{2} Z(\frac{q-1}{2}) - \frac{1}{2} Z(\frac{q+1}{2}) \\ \\
        S = \frac{3}{2} Z(\frac{q+1}{2}) - \frac{1}{2} Z(\frac{q-1}{2})
    \end{array}
\end{equation}

So,
\begin{equation}
    \frac{Y}{2} + U + V + x = \frac{S}{2} + x = \frac{1}{4} \left( Z\left(\frac{q+1}{2}\right) + Z\left(\frac{q-1}{2}\right) \right)
\end{equation}

Note $Z(H)$ is $\dim S_2(\Gamma)$ for a certain congruence subgroup $\Gamma$ and, so, is equal to the genus of the corresponding modular curve. Hence one can compute (Hecke computes this in \cite{hecke1928})
\begin{equation}
    Z\left(\frac{q+1}{2}\right) + Z\left(\frac{q-1}{2}\right) = \frac{q^2-1}{6} - (q-1)
\end{equation}

So,
\begin{equation}
\label{eqn:z_repn}
    z = \frac{y_+ - y_-}{2} + \frac{q-1}{4} + \frac{q^2-1}{24}
\end{equation}

\section{Computing $z$ using Riemann-Roch}

\subsection{Some general definitions and results}

Before we actually compute $z$ we will introduce some definitions and present some general facts about automorphic factors and its relation with modular forms. For the most part we will follow \cite{rankin77}.

Throughout this subsection, $\Gamma \subseteq \SL_2(\ZZ)$ denotes a congruence subgroup, $\overline{\Gamma}$ its image in $\PSL_2(\ZZ)$ and if $T = \begin{pmatrix} a & b \\ c & d \end{pmatrix} \in \Gamma$ and $z \in \cH$, then $T:z = cz + d$. Moreover, $P = \begin{pmatrix} 1 & 1 \\ 0 & 1 \end{pmatrix} \in \SL_2(\ZZ)$.

\begin{df}
	A map $\nu : \Gamma \times \cH \rightarrow \CC$ is called an \emph{automorphic factor of weight $k$} if
	
	\begin{enumerate}
		\item $z \mapsto \nu(T,z)$ is holomorphic for all $T \in \Gamma$
		\item $|\nu(T,z)| = |T:z|^k$ for all $T \in \Gamma$ and $z \in \cH$
		\item $\nu(ST, z) = \nu(S,Tz) \nu(T,z)$ for all $S,T \in \Gamma$ and $z \in \cH$
		\item if $-I \in \Gamma$, then $\nu(-T,z) = \nu(T,z)$ for all $T \in \Gamma$ and $z \in \cH$
	\end{enumerate}
\end{df}

Let us denote $\mu(T,z) := (T:z)^k$. One can prove that $\mu(T,z) = v(T) \mu(T,z)$, where $|v(T)| = 1$. The map $v$ is called the \emph{multiplier system} associated with $\nu$. Note that it determines $\nu$.

\begin{df}
	If $L \in \SL_2(\ZZ)$, then we define the \emph{parameter of the cusp} $L \infty$ (with respect to $\Gamma$ and $\nu$) to be the only $\kappa_L \in [0,1)$ such that $v(L P^{n_L} L) = \exp(2 \pi i \kappa_L)$ where $n_L$ is the width of the cusp $L \infty$ with respect to $\Gamma$. (One proves that this depends only on the orbit of $L \infty$ via the action of $\Gamma$).
\end{df}

\begin{df}
	An \emph{unrestricted modular form} of weight $k$ for the group $\Gamma$ with respect to the automorphic factor $\nu$ (or multiplier system $v$) is a holomorphic function $\varphi: \cH \rightarrow \CC$ such that $f(Tz) = \nu(T,z) f(z)$ for all $T \in \Gamma$ and $z \in \cH$. The space of all such functions is denoted $M'_k(\Gamma, v)$. (Note: the definition found in \cite{rankin77} is a little more general; moreover, the notation used is slightly different).
\end{df}

\begin{fact}
	Let $\varphi \in M'_k(\Gamma,v)$, $L \in \SL_2(\ZZ)$ and $n_L$ be the width of the cusp $L \infty$ with respect to $\Gamma$. Define $\varphi[L]_k (z) := \mu(L,z)^{-1} \varphi(Lz)$ and $\varphi[L]_k^*(z) := \exp(-2 \pi i \kappa_L z / n_L) \varphi[L]_k (z)$. Then
	
	\begin{enumerate}
		\item $\varphi[L]_k (z) = (z + n_L) = \exp(2 \pi i \kappa_L) \varphi[L]_k(z)$
		\item $\varphi[L]_k^*(z + n_L) = \varphi[L]_k^*(z)$
	\end{enumerate}
\end{fact}

\begin{df}
	A function $\varphi \in M'_k(\Gamma, v)$ is called a \emph{modular form} of weight $k$ for the group $\Gamma$ with respect to the automorphic factor $\nu$ (or multiplier system $v$) if $\varphi$ is holomorphic at the cusps.
\end{df}

\begin{rmk}
	A function $\varphi \in M'_k(\Gamma, v)$ is said to be \emph{holormophic at the cusps} if for all $L \in \SL_2(\ZZ)$, $\varphi[L]_k^*(z) = \mysum\limits_{m = N_L}^{\infty} a_m t^m$ for some $N_L \in \ZZ_{\geq 0}$ where $t = \exp(2 \pi i z / n_L)$.
\end{rmk}

\begin{df}
	If $\varphi \in M_k(\Gamma, v) \backslash \{ 0 \}$ and $z \in \cH^*$, then we define the \emph{order of $\varphi$ at $z$} by
	\[
		\ord(\varphi, z, \Gamma) = \left\{
										\begin{array}{lcl}
											\frac{\ord(\varphi, z)}{\overline{\Gamma}_{z}} & , & \textrm{ if } z \in \cH \\
											\kappa_L + N_L & , & \textrm{ if } z = L \infty
										\end{array}
									\right.
	\]
	(where $\ord(\varphi, z)$ is just the order of $\varphi$ at $z$ as a holomorphic function on $\cH$, $\overline{\Gamma}_{z}$ is the stabilizer of $z$ in $\overline{\Gamma}$ and $N_L$ is as in the previous remark such that $a_{N_L} \neq 0$).
\end{df}

\begin{fact} Let $\varphi \in M_k(\Gamma, v)$. Then
    \[
    	\varphi[L]_k(z) = \exp(2 \pi i \kappa_L z / n_L) \mysum\limits_{m = N_L}^{\infty} a_m(L) \exp(2 \pi i z / n_L)
    \]
    (This equality justifies the definition of the order of $\varphi$ at a cusp)
\end{fact}

\begin{df}
	Given $\varphi \in M_k(\Gamma, v) \backslash \{ 0 \}$, we define
	\[
		\ord(\varphi, \Gamma) = \mysum_{z \in \Gamma \backslash \cH^*} \ord(\varphi, z, \Gamma)
	\]
\end{df}

\begin{thm}(\textit{theorem 4.1.4 in \cite{rankin77}})
\label{thm:valence_formula}
	If $\varphi \in M_k(\Gamma, v) \backslash \{ 0 \}$ then $\ord(\varphi, \Gamma) = \dfrac{\mu k}{12}$, where $\mu = \left[ \PSL_2(\ZZ) : \overline{\Gamma} \right]$.
\end{thm}

This theorem is sometimes called `Valence Formula'.

\subsection{Computing $z$}
\label{ssec:computing_z}

We are now going to use Riemann-Roch to compute $z$. (throughout this section, we view $P = \begin{pmatrix} 1 & 1 \\ 0 & 1 \end{pmatrix}$ in $\SL_2(\ZZ)$ or in $\PSL_2(\FF_q)$ interchangeably depending on the context).

First, using that $\Gamma_1(q) = \left< \Gamma(q), P \right>$ we notice that $V = M_2(\Gamma_1(q), v)$ where $v$ is the multiplier system associated with the automorphic factor $\nu$ given by
\[
	\begin{array}{ccccc}
		\nu & : & \Gamma_1(q) \times \cH & \longrightarrow & \CC \\
		    &   & (\gamma,z) & \longmapsto & (cz+d)^2 \zeta^b
	\end{array}
\]
where $\gamma = \begin{pmatrix} a & b \\ c & d \end{pmatrix}$. So $v : \Gamma_1(q) \rightarrow \CC$ is given by $v(\gamma) = \zeta^b$.

Now fix $\varphi \in V \backslash \{ 0 \}$ and define $\widetilde{V} := \left\{ \frac{\varphi_1}{\varphi} \ \big| \ \varphi_1 \in V \right\} \subseteq \CC(X_1(q)) = $ set of rational functions on $X_1(q)$. Note that $\widetilde{V} \cong V$ (as $\CC$-vector spaces) and, hence, $z = \dim_{\CC} \widetilde{V}$.

The idea is to write $\widetilde{V} = L(D) = \{ f \in \CC(X_1(q)) \mid div(f) \geq -D \}$ for a suitable divisor $D$. We then need to find the degree of $D$ in order to apply Riemann-Roch.

$D$ is basically the ``divisor'' of $\varphi$ but first we have to understand what we mean by ``divisor'' of $\varphi$. Notice the order of $\varphi$ as defined in the previous section might not be an integer for elliptic points and cusps. So we have to deal with those points.

First note that since $q \geq 3$ then $\Gamma_1(q)$ has no elliptic elements. Hence we only need to deal with cusps. We need to find all cusps with non-zero parameter ($\kappa_L \neq 0$).

Let $\dfrac{r}{s}$ be a cusp of $\Gamma_1(q)$ (and assume $\gcd(r,s) = 1$) and $L = \begin{pmatrix} r & x \\ s & y \end{pmatrix} \in \SL_2(\ZZ)$ and assume $n = n_L$ is the width of the cusp $r/s = L \infty$. Then
\begin{equation}
\label{eqn:stab_cusp}
    L P^n L^{-1} = \begin{pmatrix} 1 - nrs & nr^2 \\ -ns^2 & 1 + nrs \end{pmatrix}
\end{equation}

Note that if $n \equiv 0 \pmod{q}$, then $\kappa_L = 0$. If $n \not\equiv 0 \pmod{q}$, then $s \equiv 0 \pmod{q}$ (otherwise $n \neq n_L$). So the cusps $r/s$ that are problematic are the ones with $s \equiv 0 \pmod{q}$. One checks that these are represented by
\begin{equation}
\label{eqn:bad_cusps}
    \dfrac{r}{q} \ , \ r = 1, 2, \dotsc, \frac{q-1}{2}
\end{equation}
(and these are not $\Gamma_1(q)$-equivalent).

Hence, by (\ref{eqn:stab_cusp}), we see that $n_L = 1$ for all the cusps in (\ref{eqn:bad_cusps}). Moreover, for those cusps, $v(L P^{n_L} L^{-1}) = \exp(2 \pi i r^2/q) = \zeta^{r^2}$ and, thus, $\kappa_L = \left\{ \frac{r^2}{q} \right\} = $ the fractional part of $r^2/q$.

The conclusion is that $\widetilde{V} = L(D)$ where $D$ is a divisor of degree $m = \mu / 6 - \mysum\limits_{r=1}^{(q-1)/2} \left\{ \frac{r^2}{q} \right\}$ where $\mu = \left[ \PSL_2(\ZZ) : \overline{\Gamma_1(q)} \right] = (q^2 - 1)/2$ (by theorem \ref{thm:valence_formula}).

\begin{claim}
    Let $g$ be the genus of $X_1(q)$ (which is known to be $g = (q-5)(q-7)/24$), then $m \geq 2g - 2$.
\end{claim}
\begin{proof}

    $m - 2g + 2 = (q^2-1)/12 - \mysum\limits_{r=1}^{(q-1)/2} \left\{ \frac{r^2}{q} \right\} - (q-5)(q-7)/12 + 2 = (12q-36)/12 - \mysum\limits_{r=1}^{(q-1)/2} \left\{ \frac{r^2}{q} \right\} + 2 = q - 1 - \mysum\limits_{r=1}^{(q-1)/2} \left\{ \frac{r^2}{q} \right\} > 0$.
\end{proof}

So Riemann-Roch gives us $z = \dim_{\CC} \widetilde{V} = m - g + 1$.

Denote $\chi(n) = \left( \frac{n}{q} \right)$ the Legendre symbol. Since $\left\{ \frac{r^2}{q} \right\} = \frac{r^2 \mod{q}}{q}$ we obtain that
\[
    \mysum\limits_{r=1}^{(q-1)/2} \left\{ \frac{r^2}{q} \right\} = \frac{1}{2} \mysum\limits_{n=1}^{q-1} \frac{n}{q}(1 + \chi(n)) = \frac{q-1}{4} + \frac{1}{2q} \mysum_{n=1}^{q-1} n \chi(n)
\]

Hence, we can compute
\[
    z = m - g + 1 = \dotsb = \frac{q^2 + 6q - 7}{24} - \frac{1}{2q} \mysum_{n=1}^{q-1} n \chi(n)
\]

Now, using Dirichlet's class number formula, we obtain
\begin{equation}
\label{eqn:z_RR}
    z = \frac{q^2 + 6q - 7}{24} - \frac{1}{2}h(-q)
\end{equation}
where $h(-q) = $ class number of $\QQ(\sqrt{-q}) / \QQ$.

\section{Conclusion and final remark}

Using equation (\ref{eqn:z_repn}) from section \ref{ssec:applying_gen_rmks} and equation (\ref{eqn:z_RR}) from section \ref{ssec:computing_z} we finally obtain what we wanted:
\[
	m_+ - m_- = h(-q)
\]

As a final remark, we note that a different proof of this result of Hecke can be found in \cite{casselmana}.

\bibliography{takei}

\begin{thebibliography}{Ran77}

\bibitem[Cas]{casselmana}
Bill Casselman.
\newblock Langlands' fundamental lemma for {$\SL(2)$}.
\newblock \url{http://www.math.ubc.ca/~cass/research/pdf/SL2.pdf}.

\bibitem[DS05]{dia&shu2005}
Fred Diamond and Jerry Shurman.
\newblock {\em A first course in modular forms}, volume 228 of {\em Graduate
  Texts in Mathematics}.
\newblock Springer-Verlag, New York, 2005.

\bibitem[FK92]{farkas&kra1992}
H.~M. Farkas and I.~Kra.
\newblock {\em Riemann surfaces}, volume~71 of {\em Graduate Texts in
  Mathematics}.
\newblock Springer-Verlag, New York, second edition, 1992.

\bibitem[Hec28]{hecke1928}
Erich Hecke.
\newblock {\"Uber ein Fundamentalproblem aus der Theorie der Elliptischen
  Modulfunktionen}.
\newblock {\em Abh. Math. Sem. Univ. Hamburg}, 6:235--257, 1928.

\bibitem[Hec30]{hecke1930}
Erich Hecke.
\newblock {\"Uber das Verhalten der Integrale 1. Gattung bei Abbildungen,
  insbesondere in der Theorie der elliptischen Modulfunktionen}.
\newblock {\em Abh. Math. Sem. Univ. Hamburg}, 8:271--281, 1930.

\bibitem[Ran77]{rankin77}
Robert~A. Rankin.
\newblock {\em Modular forms and functions}.
\newblock Cambridge University Press, Cambridge, 1977.

\end{thebibliography}
\bibliographystyle{alpha}

\end{document}